\documentclass[11pt]{article}
\usepackage{geometry,graphicx,latexsym,amssymb,verbatim, multicol}
\usepackage{amsmath, amssymb, amsthm}
\usepackage{tikz}
\usetikzlibrary{calc}

\newtheorem{thm}{Theorem}[section]
\newtheorem{lem}[thm]{Lemma}

\newtheorem{prob}{Problem}
\newtheorem{cor}[thm]{Corollary}

\newtheorem*{thma}{Theorem A}

\renewcommand{\comment}[1]{}

\renewenvironment{proof}{\noindent {\it Proof.}}{$\Box$\\}
\newcommand{\dontshow}[1]{}

\newcommand{\mad}{{\rm Mad}}
\newcommand{\bp}{{\rm b}}
\newcommand{\ov}{\overline}

\begin{document}

\begin{center}
{\LARGE  Degrees in oriented hypergraphs and Ramsey\\ $p$-chromatic number}
\mbox{}\\[8ex]

\begin{multicols}{2}

Yair Caro\\[1ex]
{\small Dept. of Mathematics and Physics\\
University of Haifa-Oranim\\
Tivon 36006, Israel\\
yacaro@kvgeva.org.il}

\columnbreak

Adriana Hansberg\\[1ex]
{\small Dep. de Matem\`atica Aplicada III\\
UPC Barcelona\\
08034 Barcelona, Spain\\
adriana.hansberg@upc.edu}\\[2ex]

\end{multicols}

\end{center}

\begin{abstract}
The family $D(k,m)$ of graphs having an orientation such that for every vertex $v \in V(G)$ either (outdegree) $\deg^+(v) \le k$ or (indegree) $\deg^-(v) \le m$ have been investigated recently in several papers because of the role $D(k,m)$ plays in the efforts to estimate the maximum directed cut in digraphs and the minimum cover of digraphs by directed cuts. Results concerning the chromatic number of graphs in the family $D(k,m)$ have been obtained via the notion of $d$-degeneracy of graphs. In this paper we consider a far reaching generalization of the family $D(k,m)$, in a complementary form, into the context of $r$-uniform hypergraphs, using a generalization of Hakimi's theorem to $r$-uniform hypergraphs and by showing some tight connections with the well known Ramsey numbers for hypergraphs.\\

\end{abstract}


\section{Introduction}

The family $D(k,m)$ of graphs having an orientation such that for every vertex $v \in V(G)$ either (outdegree) $\deg^+(v) \le k$ or (indegree) $\deg^-(v) \le m$ have been investigated recently in several papers because of the role $D(k,m)$ plays in the efforts to estimate the maximum directed cut in digraphs and the minimum cover of digraphs by directed cuts. Results concerning the chromatic number of graphs in the family $D(k,m)$ have been obtained via the notion of $d$-degeneracy of graphs (see \cite{ABGLS, BLZ, BBH, LMP, PR}). A main tool in obtaining results on $D(k,m)$ is the following well known theorem of Hakimi. For a graph $G$, the \emph{maximum average degree} of $G$ is defined as $\mad(G) = \max\frac{2|E(F)|}{|V(F)|}$, where the maximum is taken over all non-empty subgraphs $F$ of $G$.

\begin{thma}[Hakimi, \cite{Hak}]
Let $G$ be a graph. Then $G$ has an orientation such that the maximum outdegree of $G$ is at most $k$ if and only if $\mad(G) \le 2k$.
\end{thma}

In this paper we consider a far reaching generalization of the family $D(k,m)$, in a complementary form, into the context of $r$-uniform hypergraphs. To present a sample of our results we need the following definitions.

\noindent
1) Let $H$ be an $r$-uniform hypergraph. An orientation of $H$ associates with each edge an ordering of its vertices; an edge of size $r$ can be ordered in $r!$ ways. Let $D(H)$ denote an orientation of $H$. Let $P_1, P_2, \ldots, P_{{r \choose p}}$ be the $p$-sets of $\{1, . . . , r\}$, representing the possible sets of positions that gets a $p$-set contained in an edge $E$ under orientation $D(H)$. For a subset $A \subseteq V(H)$ with $|A| = p$, let $d_i(A)$ denote the number of edges in $D(H)$ in which the set of positions occupied by $A$ on the given orientation of $H$ is precisely $P_i$. When $A = \{v\}$, we simply write $d_i(v)$ for $d_i(\{v\})$, $1 \le i \le r$.  We define the {\it degree vector} of length ${r}\choose{p}$ as the vector with coordinates $d_i(A)$, $1 \le i \le {r \choose p}$, representing the number of occurrences of $A$ in each of the ${r \choose p}$ positions in which $A$ appears in the oriented edges of $D(H)$. Further, $\Delta_i(D) = \max_{v \in V} \deg_i(v)$ is the maximum among the $i$-th coordinates of the degree vectors of the vertices of $H$ under orientation $D$.

\noindent
2) Define $f(D(H),p,k)$ as the number of $p$-sets $A \subseteq V(H)$ with $d_i(A) \ge k$ for all $p$-sets $P_i \subseteq \{1,\ldots,r\}$, $1 \le i \le {r \choose p}$, and denote $f(H,p,k)$ the minimum of $f(D(H),p,k)$ over all orientations $D(H)$ of $H$. We also use $f(n,r,p,k)$ in case that $H = H(n,r) = (K_n)^r$ is the complete $r$-uniform hypergraph. For graphs, as $p =1$, we use the shorter notation $f(G,k)$ for $f(G,1,k)$. Define further $f(r,p,k)$ as the minimum $n$ such that in every orientation of the complete $r$-uniform hypergraph $H(n,r)$ there is a $p$-set with all coordinates at least $k$. Thus $f(r,p,k) = \min \{n \, : \, f(n,r,p,k) > 0 \}$.

\noindent
3) Let $H$ be an $r$-uniform hypergraph. Suppose we color the $p$-sets of $V(H)$ by some colors. An edge $E$ is $p$-monochromatic if all its $p$-sets receive the same color. The \emph{Ramsey $p$-chromatic number} $\chi_R(H,p)$ is the minimum number of colors used in coloring the $p$-sets of $V(H)$ such that no edge is $p$-monochromatic. Note that $\chi_R(H,1)$ is the traditional chromatic number $\chi(H)$ of $H$. Denote by $\chi_R(n,r,p)$ the Ramsey $p$-chromatic number of $H(n,r)$, that is, $\chi_R(n,r,p)$ is the minimum integer $t$ such that the $p$-sets of $H(n,r)$ can be colored by $t$ colors without a $p$-monochromatic edge (a monochromatic copy of $H(r,p)$). Hence, in a sense, $\chi_R(n,r,p)$ is the inverse of the Ramsey numbers. For example, $\chi_R(n,3,2) = 2$ for $3 \le n \le 5$, but $\chi_R(6,3,2) = 3$ since $R(K_3,K_3) = 6$, then $\chi_R(n, 3,2) = 3$ for $6 \le n \le 16$ but $\chi_R(17,3,2) = 4$ since $R(K_3,K_3,K_3) = 17$. Since 
\[
(*)\;\;\; c_1(321)^{k/5} \le R(K_3 : k \mbox{ colors}) \le 3k!  \;\;\; (\dagger)
\] 
we get $c_2 \log n \ge \chi_R(n,3,2) \ge c_3 \log n/ \log \log n$. The left bound (*) was given by Exoo in \cite{Ex}, while the right bound ($\dagger$) is from Chung and Grinstead \cite{CG}.

\noindent
4) Let $B(H,p)$ be a largest family of $p$-sets of $V(H)$ that can be colored using at most $r \choose p$ colors such that no edge of $H$ with all its $p$-sets in $B(H,p)$ is $p$-monochromatic.
Let $|B(H,p)| = b(H,p)$. Thus if $\chi_R(H,p) \le {r \choose p}$ then $b(H,p) = {n \choose p}$.

\noindent
5) We say that an $r$-uniform hypergraph $H$ is \emph{$r$-partite} if $V(H)$ can be partitioned into at most $r$ independent sets. Note that for an $r$-uniform hypergraph $H$, $b(H,1)$ is the cardinality of the largest induced $r$-partite subhypergraph of $H$.

\noindent
6) For an $r$-uniform hypergraph $H$, define $\mad(H) = \max \{re(F)/|F| \, : \, \emptyset \neq F \subseteq V(H)\}$.

\noindent
7) For an $r$-uniform hypergraph $H$, define $M(H,k) = \max \{ |A_1 \cup A_2 \cup \ldots \cup A_r| \}$, where the maximum is taken among all mutually vertex disjoint subsets $A_i$ such that $\mad(A_i) \le rk$, for $i = 1, 2, \ldots, r$.

\noindent
8) Recall that a hypergraph $H$ is \emph{$d$-degenerate} if in every induced subhypergraph $F$ of $H$ (including $H$ itself) there is a vertex whose degree in $F$ is at most $d$. A classical result of Szekeres and Wilf \cite{SzWi} states that, if a graph $G$ is $d$-degenerate, then $\chi(G) \le d+1$. This theorem extends easily to hypergraphs, namely if $H$ is a $d$-degenerate hypergraph then $\chi(H) \le d+1$.


\noindent
9) We shall now complete the notation used in this paper. Given a hypergraph $H$ with vertex set $V(H)$ and edge set $E(H)$, the number  edges of $H$ is denoted by $e(H)$. With $\deg(v)$ we denote the {\it degree} of $v \in V$, i.e. the number of edges containing $v$. We denote by $\delta(H)$ the minimum among all degrees of the vertices of $H$. Let $\overline{d}(H) = \frac{1}{n}\sum_{v \in V} \deg(v)$ be the \emph{average degree} of $H$. We will deal with \emph{$r$-uniform} hypergraphs, i.e. hypergraphs having every edge of size $r$. We denote with $H(n,r)$ the complete $r$-uniform hypergraph on $n$ vertices. For a subset $A \subseteq V$ of the vertex set of $H$, the {\it induced subhpergraph} $H[A]$ of $H$ by $A$ is the hyperpgraph with vertex set $A$ and all edges $E\in E(H)$ such that $E \subseteq A$. Moreover, for $v \in A$, $\deg(v \,:\, A)$ stands for the degree of $v$ in $H[A]$.
We call a subset $S \subseteq V$ {\it independent} if $|S \cap E| \le r-1$ for every edge $E \in E(H)$ and with $\alpha(H)$ we denote the maximum cardinality of an independent set of $H$. \\

\indent
All these concepts are valid also for graphs ($r = 2$), where we write $G$ in stead of $H$. Finally, we denote by $K_n$ the complete graph on $n$ vertices and by $K_{a,b}$ is the complete bipartite graph with partite sets of cardinality $a$ and $b$. \\

The following is a sample of our main results:

\noindent
1) Let $H$ be an $r$-uniform hypergraph. Then there is an orientation of $H$ such that for every vertex $v \in V(H)$, the (outdegree) $d_1(v) \le k$ if and only if $\mad(H) \le rk$.

\noindent
2) Let $H$ be an $r$-uniform hypergraph on $n$ vertices. Then $f(H,1,k) = n - M(H,k-1)$. 

\noindent
3) Let $H$ be an $r$-uniform hypergraph. Then $f(H,1,k) \ge \chi(H) - r(r(k-1) +1)$. 

\noindent
4) $f(n,r,1,k) = \max \{n- rt, 0\}$, where $t$ is the maximum integer such that ${t-1 \choose r-1} \le (k - 1)r$. 

\noindent
5) $f(H,p,1) \ge {n \choose p} - b(H,p)$ where equality holds for $p = 1$ and $p = r-1$.

\noindent
6) $f(r,r-1,1) = R(H(r,r-1),r)$, the Ramsey number of $H(r,r-1)$ using $r$ colors. In particular $f(3,2,1) = 17$, $f(4,3,1) \le 15202$ (see \cite{Ra})

\noindent
7) Let $H(n.r )$ be the complete $r$-uniform hypergraph on $n$ vertices. Suppose $r > p \ge 1$ , $t \ge p$, $k \ge 1$, and $n \ge N(r,p,t,k)$. Then in every orientation $D(H)$ there is a $t$-set $B$ of $V(H)$ all its $t \choose p$ $p$-sets having degree-vector with all coordinates at least $k$.

\noindent
8) We determine $f(G,k)$ for several families of graphs including complete $t$-partite graphs, maximal outerplanar graphs and maximal planar graphs. 

The rest of this paper is organized as follows.

\noindent
Section 2 - Generalization of Hakimi theorem to $r$-uniform hypergraphs, and complementary facts.\\ Section 3 - Bounds on $f(H,1,k)$ and $f(H,p,1)$ that will be developed in four subsections. \\[-4ex]
\begin{itemize}
\item[3.1 -] Bounds using the generalization of Hakimi theorem to $r$-uniform hypergraphs.
\item[3.2 -] Bounds using the chromatic number.
\item[3.3 -] Concrete results for families of graphs. 
\end{itemize}
Section 4 - Ramsey type theorem for $f(H,p,k)$, $p \ge 2$.\\[-4ex]
\begin{itemize}
\item[4.1 -] Bounds using the notion of $\bp(H,p)$.
\item[4.2 -] Bounds using Ramsey numbers.
\end{itemize}
Section 5 - NP-Completeness of $f(G,k)$.\\
Section 6 - Open problems.\\
Section 7 - References.

\comment{
We first need to define the following notation. For an $r$-uniform hypergraph $H$, the \emph{maximum average degree} $\mad(H)$ of $H$ is the maximum number $r\cdot e(F)/|F|$ over all non empty subsets $F$ of $V(H)$. Further, for a non-negative integer $k$, define $M(H,k) = \max \{ | A_1 \cup A_2 \cup \ldots \cup A_r |\}$, where the maximum is taken over all mutually vertex disjoint subsets $A_i$ such that $\mad_H(A_i) \le rk$, for $i=1,\ldots ,r$.

Let $D$ be an orientation of $H$, namely every edge $E$ of $H$ is assigned one of the $r!$ possible orders of its elements. For a $p$-set of vertices $A$, $1 \le p \le r -1$, we define the degree vector $v(A,D)$ of length $r \choose p$ in which every coordinate represents the number of occurrences of $A$ in one of the $r \choose p$ positions in which $A$ appears in the oriented edges of the oriented hypergraph $D$. For integers $p \ge 1$ and $k \ge 0$, define $f(D,p,k)$ as the number of $p$-sets $A$, whose vector $v(A,D)$ has all coordinates at least $k$ and denote $f(H,p,k)$ the minimum of $f(D,p,k)$ over all orientations $D$ of $H$. Evidently, $f(D,p,0) = |V(H)|$. Write $f(D,p,1) = f(D,p)$ and $f(H,p,1) = f(H,p)$ for short. We also use $f(n,r,p,k)$ in case that $H = H(n,r) = (K_n)^r$ is the complete $r$-uniform hypergraph.  Note that if $H$ is a graph, i.e. $r=2$, then $p=1$ and $f(D,k)$ is the number of vertices of $H$ whose indegree and outdegree in $D$ is at least $k$.
}


 \section{Generalization of Hakimi's Theorem to hypergraphs}

To prove the generalization of Hakimi's theorem to $r$-uniform hypergraphs, we need the following result:

\begin{lem}[Frank, Kir\'aly, Kir\'aly \cite{FKK}]\label{lem_fkk}
Let $H(V,E)$ be a hypergraph and let $f: V \rightarrow \mathbb{Z}^+$ be a mapping of the vertex set $V$ of $H$ into the set of non-negative integers. Then there is an orientation $D(H)$ of $H$ such that $\deg_1(v) = f(v)$ for every $v \in V$ if and only if
\begin{enumerate}
\item[{\rm (i)}] $\sum_{v \in V} f(v) = e(H)$ and
\item[{\rm (ii)}] $\sum_{v \in F} f(v) \ge e(F)$ for every $F \subseteq V$.
\end{enumerate}
\end{lem}

\begin{lem}\label{lemma2}
Let $H = (V,E)$ be a hypergraph and let $f: V \rightarrow \mathbb{Z}^+$ be a mapping of the vertex set $V$ of $H$ into the set of non-negative integers. Suppose that, for every $F \subseteq V$, $\sum_{v \in F}f(v) \ge e(F)$. Then there is an orientation $D(H)$ of $H$ such that $\deg_1(v) \le f(v)$ for every $v \in V$.
\end{lem}

\begin{proof}
Let $g : V \rightarrow \mathbb{Z}^+$ be a mapping such that
\begin{enumerate}
\item[(i)] for every $F \subseteq V$, $\sum_{v \in F} g(v) \ge e(F)$; 
\item[(ii)] $g(v) \le f(v)$ for every $v \in V$;
\item[(iii)] $\sum_{v \in V} g(v)$ is the minimum among all functions that satisfy (i) and (ii).
\end{enumerate}
We will show that $\sum_{v \in V} g(v) = e(H)$ and then, by Lemma \ref{lem_fkk}, there is an orientation $D(H)$ of $H$ for which $\deg_1(v) = g(v) \le f(v)$ for every $v \in V$ and we are done. Let $X \subseteq V$ be a set with maximum cardinality for which $\sum_{v \in X} g(v) = e(X)$. Possibly $X$ is the empty set.
If $X = V$ we are done, so we assume $|X| < |V|$. If $w$ is a vertex in $V \setminus X$ for which $g(w) = 0$, then by maximality of $|X|$ we obtain 
\[
\sum_{v \in X \cup \{w\}} g(v)  > e(X \cup \{w\})  \ge e(X)  = \sum_{v \in X} g(v) +0 = \sum_{v \in X} g(v) +g(w) = \sum_{v \in X \cup \{w\}} g(v),
\]
a contradiction. Hence $V\setminus X$ contains no vertex $w$ with $g(w)= 0$.
Let $z$ be in $V\setminus X$ such that $g(z)>0$ and define $h: V \rightarrow \mathbb{Z}^+$ such that $h(v) = g(v)$ for $v \in V \setminus \{z\}$ and $h(z) = g(z) -1$. Suppose now that there is a subset $F \subseteq V$ such that $\sum_{v \in F} h(v) < e(F)$. Then $z \in F$ and $e(F) \le \sum_{v \in F} g(v) =  1+ \sum_{v \in F} h(v) < 1+ e(F)$, which implies that $\sum_{v \in F} g(v) = e(F)$. But, by the maximality of $|X|$ and since $F$ is not contained in $X$ as $z \in F$ and $z \in V \setminus X$, we have that
\[
\sum_{v \in F \cup X} g(v) >  e(F \cup X) \ge e(F) + e(X)  =  \sum_{v \in F} g(v) + \sum_{v \in X} g(v)  \ge \sum_{v \in F\cup X} g(v),
\]
again a contradiction. It follows that $\sum_{v \in F} h(v) \ge e(F)$ for all $F \subseteq V$ and thus $h$ is a function satisfying (i). Since evidently $h(v) \le g(v) \le f(v)$ for all $v \in V$, $h$ also satisfies (ii). However $\sum_{v \in V} h(v) < \sum_{v \in V} g(v)$, which contradicts the minimality of $g$. Hence $|X| = |V|$ and we are done.
 \end{proof}
 
 \begin{thm}[Generalization of Hakimi's Theorem to hypergraphs]\label{Hakimi_hyper}
Let $H$ be an $r$-uniform hypergraph and $k \ge 0$ an integer. Then there is an orientation of $H$ such that, for all $v \in V$, $\deg_1(v) \le k$ if and only if $\mad(H) \le rk$. 
 \end{thm}
 
 \begin{proof}
We proceed proving first the necessity. If $\mad(H) > rk$,  then let $X$ be a subset of $V$ such that $\mad(H) = re(X)/|X| > rk$. Let $D^*$ be the induced orientation from $V$ on $X$. If $\deg_1(v) \le k$ for every $v \in V$, then in particular $\deg_1(v:X) \le k$ for every $v \in X$.
But then $e(X) = \sum_{v \in X} \deg_1(v : X) \le \sum_{v \in X} \deg_1(v) \le k |X|$. Hence multiplying by $r$ we get $re(X) \le rk |X|$ and hence $r e(X)/|X| \le rk$, a contradiction.

For the sufficiency, let $f(v) = k$ for every $v \in V$. Since $\mad(H) \le r k$, we get $r e(F)/|F| \le rk$ for every $F \subseteq V$ and $e(F) \le k |F| = \sum_{v \in F} f(v)$. Hence by Lemma \ref{lemma2} there is an orientation of $H$ such that $\deg_1(v) \le k$ for every $v \in V$. 
 \end{proof}

 
 \section{Bounds on $f(H,1,k)$ and $f(H,p,1)$}
 
  \subsection{Bounds using the generalization of Hakimi's Theorem to hypergraphs}
 
 The next theorem reveals a basic relation between $f(H,1,k)$ and $M(H,k-1)$.
 
\begin{thm}\label{f in M}
Let $H$ be an $r$-uniform hypergraph on $n$ vertices. Then $f(H,1,k) = n - M(H,k-1)$.
\end{thm}

\begin{proof}
For the upper bound, let $A_1, \ldots, A_r$ be $r$ mutually vertex disjoint sets realizing $M(H,k-1)$ and let $B = V \setminus \bigcup_{i=1}^r A_i$. Since $\mad(A_i) \le r(k-1)$, then by Theorem \ref{Hakimi_hyper}, we can orient the edges in $A_i$ such that the $\deg_i(v) \le k-1$ for all $v \in A_i$. Further, orient the edges in $B$ arbitrarily. Now we have to take care for edges which are not contained in some $A_i$ or $B$ without violating $\deg_i \le k - 1$ in $A_i$,  $i = 1, \ldots, r$. Let $E$ be an edge such that $E \setminus A_i \neq \emptyset$,  and such that $E \setminus B \neq \emptyset$ (i.e. such that it is neither contained in any $A_i$ nor in $B$). Set $a_i = |E \cap A_i|$ and $b = |E \cap B|$. Clearly, $b+ \sum_{1 \le i \le r} a_i =r$ and there are at least two positive summands and all $a_i \le r - 1$. Define a bipartite graph $T$ with one side the vertex set $X = \{1,\ldots ,r\}$ the other side the vertex set $E$ and the edges such that, for $v \in E$ and $i \in X$, $vi$ is an edge if and only if $v \notin A_i$. If we show that a perfect matching exists in this bipartite graph $T$, then this perfect matching supplies an order on $E$ such that the vertices in $A_i$ do never get position $i$ and so $\deg_i(v) \le k-1$ for all $v \in A_i$. For a subset $Q$ of $E$ consider the following cases. If $|Q| = r$, then $Q=E$ and $Q$ contains vertices either from some $A_i$ and $A_ j$ or from some $A_i$ and $B$ and in both cases $|N(Q)| = r = |Q|$. If otherwise $|Q| \le r-1$, then, as every vertex in $E$ has at least $r-1$ neighbors in $X$, it follows clearly that $|N(Q)| \ge r - 1 \ge |Q|$ and we are done. \\
Hence there is an order on $E$ that does not violate $\deg_i(v) \le k - 1$ for each $v \in A_i$, $i= 1, \ldots, r$ and we are done. Hence there are at least $M(H,k-1)$ vertices in which, for some $i = 1, \ldots ,r$, $\deg_i(v) \le k-1$, proving $n - M(H,k-1) \ge f(H,1,k)$.

For the lower bound, let $D$ be an orientation of $H$ that realizes $f(H,1,k)$. Let $A_1$ be the set of all vertices $v$ with $\deg_1(v) \le k-1$ and, for $i= 2, \ldots, r$, let $A_i$ be the set of vertices $v$ not in $\bigcup_{j=1}^{i-1} A_j$ with $\deg_i(v) \le k-1$. Consider the induced orientation $D_i$ on $H[A_i]$. Since $\Delta_i(D_i) \le k - 1$, then by Theorem \ref{Hakimi_hyper}, it follows that $\mad(A_i) \le r(k-1)$. Since the sets $A_i$, $1 \le i \le r$, are pairwise vertex disjoint, it follows that $|\bigcup_{i= 1}^{r} A_i | \le M(H,k-1)$. Hence $f(H,1,k) = n - |\bigcup_{i= 1}^{r} A_i | \ge n - M(G,k-1)$.

Combining the upper and the lower bound we obtain $f(H,1,k) = n - M(H,k-1)$.
\end{proof}

\comment{
\begin{cor}\label{K_n}
For the complete graph on $n \ge 4k-2$ vertices, $f(K_n,k) = n - 4k +2$.
\end{cor}

\begin{proof}
Observe that in $K_n$, the maximum size $|A|$ among all subsets $A$ such that $\mad(A) \le 2k-2$ is exactly $2k-1$ and so if $n \ge 4k-2$ we have $M(G,k-1)=4k-2$ (take two such classes) and with Theorem \ref{f in M} we get $f(K_n,k) = n - 4k +2$.
\end{proof}
}

\begin{thm}\label{f(n,r,1,k)}
Let $t$ be the maximum integer such that ${t-1 \choose r-1} \le (k-1) r$. Then 
$$f(n, r,1,k) = \max\{n-rt, 0\}.$$ 
In particular, $f(K_n,k) = f(n,2,1,k) = n - 4k + 2$, $f(n,3,1,k) = n - 3 \lfloor \frac{\sqrt{24k-23}+3}{2} \rfloor$ and $f(n,r,1,k) < n - e^{-1}(r-1) r^{\frac{r}{r-1}} (k-1)^{\frac{1}{r-1}}$.
\end{thm}

\begin{proof}
Let $H = H(n,r)$. Let $t$ be the maximum integer such that ${t-1 \choose r-1} \le (k-1) r$. If $A \subseteq V(H)$ is a set such that $\mad(A) \le r(k-1)$, then $\frac{r}{|A|} {|A| \choose r} \le r (k-1)$, which is equivalent to ${|A|-1 \choose r-1} \le r(k-1)$. Hence $|A| \le t$ and the maximum cardinality of such a set $A$ with $\mad(A) \le r(k-1)$ is precisely $t$. Thus $M(H(n,r),k-1) = \min \{r t, n\}$. Now Theorem \ref{f in M} yields $f(n,r,1,k) = \max \{n - rt, 0\}$.

When $r=2$, $t$ is easily computed to $2k-1$, while when $r=3$, it is not difficult to check that $t = \lfloor \frac{\sqrt{24k-23}+3}{2} \rfloor$. Finally, using $r(k-1) < {t \choose r-1} \le (\frac{t e}{r-1})^{r-1}$, we obtain that $t >  (r-1) e^{-1} (r (k-1))^{\frac{1}{r-1}}$, implying that $f(n,r,1,k) = n - rt < n - e^{-1}(r-1) r^{\frac{r}{r-1}} (k-1)^{\frac{1}{r-1}}$.
\end{proof}

 \begin{cor}
Let $G$ be a graph on $n$ vertices with minimum degree $\delta(G) \ge (t-1)\frac{n}{t}$. Then $f(G,k) \ge (t-4k +2) \lfloor\frac{n}{t}\rfloor$.
\end{cor}

\begin{proof}
Since $\delta(G) \ge ( t-1) \frac{n}{t}$ it follows from the Hajnal-Szemeredi Theorem \cite{HajSm} that $G$ has $\lfloor\frac{n}{t}\rfloor$ vertex-disjoint copies of $K_t$. Each copy of $K_t$ supplies, by Theorem \ref{f(n,r,1,k)}, at least $t - 4k +2$ vertices with indegree and outdegree at least $k$. Hence we have at least $(t-4k +2) \lfloor \frac{n}{t} \rfloor$ vertices with indegree and outdegree at least $k$.
\end{proof}

 Theorem \ref{f in M} allows us to deduce a Turan's type result for the maximum number of edges in an $r$-uniform hypergraph $H$ with $f(H,1,k) = 0$.

\begin{thm}
Let $H$ be an $r$-uniform hypergraph with $f(H,1,k) = 0$. Then $$e(H) \le {n \choose r} - r {n/r \choose r} + (k-1) n$$
and this bound is sharp for $n > (k-1)r^{2r-1}$ when $r^2$ divides $n$.
\end{thm}

\begin{proof}
Since $f(H,1,k) = 0$, Theorem \ref{f in M} implies that $M(H,k-1) = n$. Let now $A_1 \cup A_2 \cup \ldots \cup A_r$ be $r$ vertex disjoint sets with $\mad(A_i) \le r(k-1)$, $1 \le i \le r$, realizing $M(H,k-1)$. Then $|A_1 \cup A_2 \cup \ldots \cup A_r| = n$. Then, by convexity, we have that $\sum_{i=1}^r {|A_i| \choose r} \ge r {n/r \choose r}$. Moreover, as $\frac{r e(A_i)}{|A_i|}\le \mad(A_i) \le r (k-1)$, it follows that $e(A_i) \le (k-1) |A_i|$ for all $i =1, 2, \ldots, r$. Hence the number of edges of $H$ is at most
\[
e(H) \le {n \choose r} - \sum_{i=1}^r {|A_i| \choose r} + \sum_{i=1}^r e(A_i) 
\le {n \choose r} - r {n/r \choose r} + (k-1) n.
\]
To see the sharpness, let $r^2$ divide $n$ and take $|A_i| = n/r$ for , $i = 1,2, \ldots, r$. Then $r$ divides $|A_i|$. By the well-known Theorem of Baranyai \cite{Ba}, the $r$-uniform hypergraph induced on $A_i$ has a $1$-factorization. In particular, take precisely $r(k-1)$ $1$-factors. Each $1$-factor contributes with $|A_i|/r$ edges and so we obtain exactly $r(k-1)|A_i|/r = (k-1)|A_i|$ edges in each $A_i$. Altogether, we obtain $(k-1)n$ edges and thus we have equality in the inequality given above. Taking $r(k-1)$ $1$-factors is possible if $r(k-1) \le  \frac{e(A_i)}{|A_i|/r} = {n/r \choose r} \frac{r^2}{n}$ (which gives the total number of $1$-factors). Hence $r(k-1) \le \frac{r^2}{n} {n/r \choose r}$, which gives $n(k-1) < r {n/r \choose r} < r{(\frac{e n}{r^2})}^r$. Taking logarithm we obtain $\log(n) + \log(k-1) <   \log r + r \log(\frac{e n}{r^2})  = \log r + r \log (en) - 2r \log r = r + r \log n  - (2r-1) \log r$. Hence $\log(k-1) < (r-1) \log n + r - (2r-1) \log r$. Rearranging we get  $\frac{\log(k-1)  - r + (2r-1)\log r}{r-1} <  \log n$. This is indeed fulfilled when $n > (k-1)r^{2r-1}$, since then $\log n > \log(k-1)r^{2r-1} = \log(k-1) + \log r^{2r-1} > \frac{\log(k-1)  - r + (2r-1)\log r}{r-1}$.
\end{proof}

In order to get more information from Theorem \ref{f in M} we need the following technical lemmas.
 
\begin{lem}\label{lemma3}
Let $H$ be an $r$-uniform hypergraph on $n$ vertices with $\mad(H) \le k$, where $k$ is a non-negative integer. Then:\\[-4ex]
\begin{enumerate}
\item[\rm{(1)}] For every subhypergraph $F$ of $H$, $\mad(F) \le k$. 
\item[\rm{(2)}] $H$ is $k$-degenerate.
\item[\rm{(3)}] $H$ is $(k +1)$-colorable. 
\item[\rm{(4)}] $\alpha(H) \ge n/(k+1)$.
\end{enumerate}
 \end{lem}
 
\begin{proof}
(1) This is evident from the definition of $\mad(H)$. \\
(2) $H$ is $k$-degenerate since in every subhypergraph $F$ of $H$ (including H) $\delta(F) \le \ov{d}(F) \le \mad(F) \le \mad(H) \le k$.\\
(3) By the Szekeres-Wilf Theorem for hypergraphs (see \cite{SzWi}, the same proof as for graphs) if $H$ is $d$-degenerate, then $\chi(H)\le d +1$. Hence in our case $\chi(H) \le k +1$.\\
(4) Observe that $\alpha(H) \ge \frac{n}{\chi(H)} \ge \frac{n}{k +1}$.
\end{proof}

\begin{lem}\label{lemma4} Let $k \ge 0$ be an integer. The following assertions hold. \\[-4ex]
\begin{enumerate}
\item[\rm{(1)}] If $H$ is a $k$-degenerate $r$-uniform hypergraph, then $\mad(H) \le rk$.
\item[\rm{(2)}] Suppose $H$ is an $(r(k+1) - 1)$-degenerate hypergraph. Then $V(H)$ can be partitioned into $r$ vertex disjoint subsets $V(H) = \bigcup_{i=1}^r A_i$ such that the induced subhypergraph on $A_i$ is $k$-degenerate for each $1 \le i \le r$.
\end{enumerate}
\end{lem}

\begin{proof}
(1) Since $H$ is $k$-degenerate, so does every subhypergraph $F$ of $H$. Hence for every subhypergraph $F$ we have by induction $e(F) \le k |F|$ and hence $r \; e(F)/|F| \le rk$ and $\mad(H) \le rk$.\\
(2) Suppose first that $|V(H)| \le r$. Clearly $H$ has at most one edge and is $1$-degenerate and as $1 \le r(0+1)-1 = r-1$, $H$ is $(r-1)$-degenerate and we can partition $V(H)$ into singletons which are $0$-degenerate. Suppose we have proven the result for $(r(k+1)-1)$-degenerate hypergraphs of order $n$. Let now $H$ be an at most $(r(k+1)-1)$-degenerate $r$-uniform hypergraph of order $n+1$. By $(r(k+1)-1)$-degeneracy and Lemma \ref{lemma3}(2), there is a vertex $v$ with $\deg(v) \le r(k+1) -1$ and such that $H^* = H - v$ is also $(r(k+1) - 1)$-degenerate. By induction, $V(H^*)$ can be partitioned into $r$ vertex disjoint subsets $A_i$, $1 \le i \le r$, all of them inducing $k$-degenerate subhypergraphs. Since $\deg(v) \le r(k+1)-1$, there is at least one $A_i$ sharing at most $k$ edges with $v$. But then $A_i \cup \{v\}$ is again a $k$-degenerate subhypergraph.
 \end{proof}

For the following, define $\beta_d(H)$ as the maximum cardinality $|F|$ over all subsets $F \subseteq V(H)$ such that the induced subhypergraph on $F$ is $d$-degenerate. Note that $\beta_0(H) = \alpha(H)$.
 
\begin{thm}\label{Thm3}
Let $H$ be an $r$-uniform hypergraph on $n$ vertices and $G$ a graph. Then the following assertions hold.
\begin{enumerate}
\item[\rm{(1)}] $f(H,1,k) = n - M(H,k-1) \ge n - \alpha(H)r(rk -r+1)$
\item[\rm{(2)}] $n - \beta_{rk - 1}(H) \ge f(H,1,k) \ge n- r \beta_{r(k-1)}(H)$. 
\item[\rm{(3)}] If G has average degree $\ov{d} \ge 4k - 2$, then $f(G,k) \le \frac{\ov{d} - 2k +1}{\ov{d} +1} n$.
\end{enumerate}
\end{thm}
 
\begin{proof}
(1) Let $A_i$, $i = 1, \ldots, r$ be $r$ vertex disjoint classes such that $\mad(A_i) \le r(k - 1)$ and $|\bigcup_{i=1}^r A_i| = M(H,k-1)$. By Lemma \ref{lemma3}(4), $A_i$ contains an independent set $X_i$ of cardinality $|X_i| \ge \frac{|A_ i|}{r(k-1) +1} = \frac{|A_i|}{rk - r +1}$. Clearly,
\[
r \alpha(H) \ge \sum_{i=1}^r |X_i| \ge \sum_{i=1}^r \frac{|A_i|}{rk - r +1} = \frac{M(H,k-1)}{rk - r +1}.
\]
Hence, $\alpha(H)r(rk- r +1) \ge M(H,k-1)$. From this we get 
\[
f(H,1,k) = n - M(H,k-1) \ge n - \alpha(H)r(rk- r+1)\]
(2) We have to show that $\beta_{rk-1}(H) \le M(H,k-1) \le r \beta_{r(k-1)}(H)$. \\
(i) Let $F \subseteq V(H)$ be a set such that its induced subhypergraph of $H$ is $(rk - 1)$-degenerate  and such that $|F| = \beta_{rk-1}(H)$. By Lemma \ref{lemma4}(2), $V(H)$ can be partitioned into $r$ vertex disjoint subsets $A_i$ each of them inducing a $(k-1)$-degenerate subhypergraph. By Lemma \ref{lemma4}(1), $\mad(A_i) \le r(k-1)$ and hence $|F| = |\bigcup_{i-1}^r A_i | \le M(H,k-1)$. \\
(ii) Let $A_i $, $i = 1, \ldots, r$ be $r$ vertex disjoint classes such that $\mad(A_i) \le r(k-1)$ and $\bigcup_{i=1}^r A_i| = M(H,k-1)$. By Lemma \ref{lemma3}(2), $A_i$ is $r(k-1)$-degenerate and hence $|A_i| \le \beta_{r(k-1)}(H)$ and $M(G,k-1) = |\cup A_i| \le r \beta_{r(k-1)}(H)$.
Hence with (i) and (ii) and Theorem \ref{f in M}, we have
\[
n - r \beta_{r(k-1)}(H) \le f(H,1,k) = n - M(H, k-1) \le n - \beta_{rk-1}(H)
\]
and we are done.\\
(3) By a result given in \cite{AKS}, if $\ov{d}(G) \ge 2d$, then $\beta_d(G) \ge \frac{d+1}{\ov{d} +1}n$. 
In our case set $d = 2k-1$, so if $\ov{d} \ge 2(2k-1) = 4k - 2$, then $\beta_{2k-1}(G) \ge 2kn/(\ov{d} +1)$.
Hence, by Theorem \ref{Thm3}(2), \[
f(G,k) \le n - \beta_{2k-1}(G) \le n-\frac{2kn}{\ov{d}+1} = \frac{\ov{d} - 2k +1}{\ov{d} +1} n.
\]
\end{proof}

\subsection{Lower bound using the chromatic number}

Theorem \ref{low_chi} is a generalization for $r$-uniform hypergraphs of the case $r=2$ proved in \cite{ABGLS}.

 \begin{thm}\label{low_chi}
Let $H$ be an $r$-uniform hypergraph on $n$ vertices. Then $f(H,1,k) \ge \chi(H) - r (r(k-1)+1)$. This bound is sharp in case that $r=2$.
\end{thm} 

\begin{proof}
Let $H$ be a $r$-uniform hypergraph with chromatic number $\chi(H) \ge r(r(k-1)+1)$, otherwise there is nothing to prove. We will show first that in any orientation $D(H)$ of $H$ there is at least one vertex $v$ with $\deg_i(v) \ge k$ for $i = 1, 2, \ldots, r$. Suppose there is an orientation $D = D(H)$ of $H$ such that for every vertex $x \in V$ there is an index $j$ such that $\deg_j(x) < k$. Let $A_1$ be the set of vertices $u$ with $\deg_1(u) <k$. Define recursively $A_j$ as the set of vertices in $u \in V \setminus \cup_{i=1}^{j-1} A_i$ with $\deg_j(u) <k$, where some $A_i$ are possibly empty. If $A_i$ is not empty, then for all vertices $v \in A_i$ we have that $\deg_i(v) <k$, in the induced orientation of the $r$-uniform subhypergraph induced by $A_i$ still all vertices $v \in A_i$ have $\deg_i(v) < k$. By Theorem \ref{Hakimi_hyper}, this is possible if and only if $\mad(A_i) \le r(k-1)$. By Lemma \ref{lemma3} (3), it follows that $H[A_i]$ is $r(k-1)+1$-colorable. Hence we can take for each nonempty $A_i$ $r(k-1)+1$ different and new colors and we obtain $\chi(H) \le r (r(k-1)+1)$, which is a contradiction. 

So, every orientation $D(H)$ of $H$ contains a vertex $v$ with $\deg_i(v) \ge k$ for all $i = 1, 2, \ldots, r$. Now let us prove the theorem by induction on $\chi(H)$. The case $k = 1$ is easy, so we assume that $k \ge 2$. For $\chi(H) = r (r(k-1)+1) +1$, as we already showed, there is indeed one vertex, say $u$, with $\deg_i(u) \ge k$ for all $i = 1, 2, \ldots, r$. Now assume we have proved the theorem for $\chi(H) = t \ge r (r(k-1)+1) +1$. Let now $H$ have chromatic number $\chi(H) = t+1$ and let $D(H)$ be any orientation of $H$. Consider $H^* = H - u$, where evidently $\chi(H^*) \ge t$. If $\chi(H^*) = t$, then by the induction hypothesis in every orientation of $H^*$, including the orientation induced by $D(H)$, there are at least $t - r (r(k-1)+1)$ vertices with all coordinates of their degree vector at least $k$. Then, together with $u$, there are at least $t -r (r(k-1)+1) +1 = (t+1) - r (r(k-1)+1)$ vertices with all coordinates at least $k$ and we are done. Hence suppose that $\chi(H^*) = t+1$. Then again there is a vertex $u' \in V(H^*)$ with all its coordinates at least $k$ and we may consider $H^{**} = H^* - u'$ and repeat. Hence we have proved that there are at least $\chi(H) - r (r(k-1)+1)$ vertices with all coordinates at least $k$ and we obtain $f(H,1,k) \ge \chi(H) - r (r(k-1)+1)$.

To prove the sharpness for $r=2$, consider the complete graph on $n$ vertices $K_n$, where clearly $\chi(K_n) = n$ and, by Theorem \ref{f(n,r,1,k)}, $f(K_n,k) = n - 4k + 2$. Hence $f(K_n,k) = n - 4k + 2 = \chi(H) - 2 (2(k-1)+1)$.
\end{proof}

\comment{
We will construct an orientation with exactly $n - 4k+2$ vertices of indegree and outdegree at least $k$. Let $A$ and $B$ be two disjoint sets each of cardinality $2k-1$ and let $C$ be another set of $n - 4k +2$ vertices. These sets form the vertex set of $K_n$. Orient all edges from $A$ to $B$. Consider $A$ and $B$ as $K_{2k-1}$ and apply eulerian orientation (equitable) such that inside $A$ and inside $B$ the outdegree and the indegree are equal to $k-1$. In the resulting orientation of $K_{4k-2}$ vertices of $A$ have degree vector $(3k-2, k-1)$ and vertices of $B$ have degree vector $(k-1, 3k-2)$. Orient all edges from vertices of $A$ to vertices of $C$ and all edges from vertices of $C$ to vertices of $B$ and orient inside $C$ arbitrarily. In this orientation of $K_n$ the vertices of $A$ have indegree $k-1$, the vertices of $B$ have outdegree $k-1$ and hence at most $n - 4k +2$ have both indegree and outdegree at least $k$. Hence $f(K_n,k) =n-4k+2$.
}

\subsection{Concrete results for families of graphs}

Define $\alpha_2(G) = \max { |A \cup B|}$, where the maximum is taken over all vertex disjoint independent sets $A$ and $B$. Clearly, $\alpha_2(G) = M(G,0)$. Observe that trivially $2 \alpha(G) \ge \alpha_2(G) \ge \alpha(G)$ with the lower bound attained if and only if $G = \ov{K_n}$. With this definition we have the following corollary to Theorem \ref{Thm3}.

\begin{cor}\label{coro_graphs}
Let $G$ be a graph on $n$ vertices. Then the following holds.\\[-4ex]
\begin{enumerate}
\item[{\rm(1)}] $n - \alpha(G) \ge f(G,1) \ge n - 2 \alpha(G)$ and both bounds are sharp.
\item[{\rm(2)}] If there are two vertex-disjoint independent sets of cardinality $\alpha(G)$, then $f(G,1) = n - 2 \alpha(G)$. 
\item[{\rm(3)}] If $\chi(G) = t$, then $f(G,1) \le \lfloor \frac{(t-2)}{t}n \rfloor$ and this bound is sharp.
 \end{enumerate}
\end{cor}

\begin{proof}
(1) Observe that trivially $2 \alpha(G) \ge \alpha_2(G) \ge \alpha(G)$ and since $\alpha_2(G) = M(G,0)$, substituting in Theorem \ref{Thm3}, we get the desired result. Another proof of the lower bound is as follows. The Gallai-Milgram Theorem (see \cite{GaMi}) states that every oriented graph $G$ with independence number $\alpha(G)$ has a vertex partition into $\alpha(G)$ oriented paths. So let $P_1, P_2, \ldots ,P_{\alpha(G)}$ be such a partition. Let $P_1, \ldots ,P_t$ be the paths consisting of a unique vertex and $P_{t+1}, \ldots, P_{\alpha(G)}$ the paths consisting each of at least two vertices. Clearly, each of the paths $P_{t+1}, \ldots, P_{\alpha(G)}$ contribute with $|P_i|-2$ vertices whose coordinates are at least $1$ (those vertices except the head and the tail of the path). Hence, in every orientation $D$ of $G$, there are at least $\sum_{i = t+1}^{\alpha(G)} |P_i| - 2$ vertices with coordinates at least $1$. However,
\begin{eqnarray*}
\sum_{i = t+1}^{\alpha(G)} |P_i| - 2 &=& |P_1|+ \ldots +|P_t| - t+ \sum_{i = t+1}^{\alpha(G)} |P_i|-2\\
                                                              &=&  \sum_{i = 1}^{\alpha(G)} |P_i| -2( \alpha(G)-t)-t\\
                                                              &=& n-2 \alpha(G)+ t \ge n- 2 \alpha(G).
 \end{eqnarray*}
So, in every orientation of $G$ there are at least $n - 2 \alpha(G)$ vertices with indegree and outdegree at least $1$.

Note that the upper bound is attained if and only if $G = \ov{K_n}$. Otherwise, if $G$ has size at least one, then $n - \alpha(G) - 1\ge f(G,1)$. Further, the lower bound is attained for example by the graph $G = tK_m$ where $\alpha(G) = t$ and $f(G,1) = t(m - 2)$, since by Theorem \ref{f(n,r,1,k)} $f(K_m,1) = m-2$. Hence $f(G,1) = tm - 2t = n - 2 \alpha(G)$.

(2) In this case $M(G,0) = 2 \alpha(G)$ and the result follows. 

(3) Suppose that $t \ge 2$, otherwise there is nothing to prove. Let $A_1, \ldots, A_t$ be the color classes  such that $|A_1| \ge \ldots \ge |A_t|$. Then $|A_1| \ge \frac{n}{t}$ and $|A_2| \ge \frac{n - |A_1|}{t-1}$. Hence $\alpha_2(G) \ge |A_1 \cup A_2| \ge |A_1| + \frac{n- |A_1|}{t-1} = \frac{(t-2)|A_1| +n }{t-1}$ and then $n - \frac{(t-2)|A_1| +n }{t-1} \ge n - \alpha_2(G) = f(G,1)$. This implies that $\frac{(n - |A_1|)(t-2)}{t-1} \ge f(G,1)$ and, since$|A_1| \ge \frac{n}{t}$, we obtain 
\[
\frac{(n - n/t)(t-2)}{t-1} = \frac{(t-1)n(t-2)}{(t-1)t} = \frac{n(t-2)}{t}\ge f(G,1)
\]
and hence  $f(G,1) \le \lfloor (t-2)\frac{n}{t} \rfloor$.
Another proof of this upper bound is as follows. Let $A_1, \ldots ,A_t$ be the partition of $V$ into $\chi(G)= t$ independent sets and assume $|A_1| \ge \ldots \ge |A_t|$. Orient all edges from $A_1$ to $V \setminus A_1$ and all the edges into $A_2$ from $V \setminus A_2$. Clearly, the vertices in $A_1$ have indegree exactly $0$ and the vertices in $A_2$ have outdegree exactly $0$. Clearly, $|A_ 1 \cup A_2| \ge 2 \frac{n}{t}$ and hence in the above orientation there are at most $(t-2)n/t$ vertices with indegree and outdegree at least $1$.

That this bound is sharp can be seen by the complete $t$-partite graph with all parts equal to $n/t$. Then clearly the graph has a $K_t$-factor containing $n/t$ vertex-disjoint copies of $K_t$. In every orientation of $G$ every copy of $K_t$ supplies $t - 2$ vertices with indegree and outdegree at least $1$. Altogether we have (by vertex disjointness) at least $\lfloor (t-2)n/t \rfloor$ vertices with indegree and outdegree at least $1$.
\end{proof}


If $G = K_{n_1,n_2,\ldots, n_t}$, i.e. a a complete $t$-partite graph with partition sets $V_i$ of cardinality $|V_i| = n_i$, $1 \le i \le t$, then we are able to compute $f(G,k)$. For this purpose we will first prove the following lemma.

\begin{lem}\label{lem_t-partite}
Let $G$ be a complete $t$-partite graph $K_{n_1,n_2,\ldots, n_t}$ with partition sets $V_i$ of cardinality $|V_i| = n_i$ for $1 \le i \le t$, $t \ge2$, where $n_1 \ge n_2 \ge \ldots \ge n_t \ge 1$. Among all sets $A \subseteq \bigcup_{i=1}^t V_i = V$ of cardinality $a$, where $0 < a < n = \sum_{i=1}^t n_i$, the minimum number of edges in the subgraph induced by $A$ is
\[
\min_{A \subseteq V, |A|=a} e(G[A]) = \sum_{1 \le i < j \le q-1} n_in_j + \left(a - \sum_{i=1}^{q-1} n_i\right)  \sum_{i=1}^{q-1} n_i,
\]
where $q$ is the index for which $ \sum_{i=1}^{q-1} n_i \le a < \sum_{i=1}^q n_i$.
\end{lem}

\begin{proof}
Let $A^* \subseteq \bigcup_{i=1}^t V_i$ be a set with minimum number of edges $e(G[A^*])$ among all subsets of cardinality $a$ and let $a_i = |V_i \cap A^*|$ and let $a_i = |V_i \cap A^*|$. Since $e(G[A^*]) = \sum_{1 \le i < j \le t} a_ia_j$, we can choose a set $A^*$ such that $a_1 \ge a_2 \ge \ldots \ge a_t$. Now let $q$ be the minimum index $i$ for which $a_i \neq n_i$ and let $s$ be the maximum index $i$ for which $a_i \neq 0$. Evidently, $s-1 \le q \le s+1$. If $q=s$ or if $s+1 = q$, we are done. So suppose that $q < s$. Since $a_q \neq n_q$ and $a_s \neq 0$, we can take vertices $x \in V_q \setminus A_q$ and $y \in A_s$. Define now $A'_q = A_q \cup \{x\}$, and $A'_s = A_s \setminus \{y\}$ and $A'_i = A_i$ for all $i \neq q, s$. Let $A' = \cup_{i=1}^t A'_i$ and $a'_i = |A_i|$. The we have
 \begin{eqnarray*}
 e(G[A']) & = & \sum_{1 \le i < j \le t} a'_i a'_j \\
                & = & a'_q \sum_{j \neq q, s} a'_j + a'_q a'_s + a'_s \sum_{j \neq q, s} a'_j + \sum_{i < j, i,j \neq q,s} a'_i a'_j\\
                & = & (a_q + 1) \sum_{j \neq q, s} a_j + (a_q+1) (a_s -1) + (a_s - 1) \sum_{j \neq q, s} a_j + \sum_{i < j, i,j \neq q,s} a_i a_j\\
                & = & a_s + a_q - 1 + \sum_{1 \le i < j \le t} a_i a_j = a_s + a_q - 1 + e(G[A]) 
 \end{eqnarray*}
 But, since $q < s$, we have that $a_s - a_q \le 0$ and hence $e(G[A']) < e(G[A])$, which is a contradiction to the minimality of $e(G[A^*])$. Hence, $q=s$ or $s+1 = q$ and thus $a_1 = n_i$ for $i \le q-1$, $a_i < n_i$ for $i = q$, and $a_i = 0$ for $i \ge q+1$. This implies that
\begin{eqnarray*}
e(G[A^*]) =  \sum_{1 \le i < j \le q-1} n_in_j + a_q  \sum_{i=1}^{q-1} n_i.
\end{eqnarray*}
Further, as $a_q = a - \sum_{i=1}^{q-1} n_i$ and $q$ is exactly the index for which $ \sum_{i=1}^{q-1} n_i \le a < \sum_{i=1}^q n_i$, we obtain that
\begin{eqnarray*}
e(G[A^*]) 
      & =& \sum_{1 \le i < j \le q-1} n_in_j + \left(a - \sum_{i=1}^{q-1} n_i\right)  \sum_{i=1}^{q-1} n_i, 
      \end{eqnarray*}
for every set $A$ of cardinality $a$. 
\end{proof}

\begin{cor}\label{cor_t-partite}
Let $G$ be a complete $t$-partite graph $K_{n_1,n_2,\ldots, n_t}$ with partition sets $V_i$ of cardinality $|V_i| = n_i$, where $1 \le i \le t$, $t \ge2$, and $n_1 \ge n_2 \ge \ldots \ge n_t \ge 1$. If $A \subseteq \bigcup_{i=1}^t V_i$ is a set of cardinality $a$ such that $ \sum_{i=1}^{q-1} n_i \le a < \sum_{i=1}^q n_i$ for an integer $q \ge 1$, then, for any $\ell \in \{1,2,\ldots, q-1\}$
\[
e(G[A]) \ge (a - n_{\ell}) n_{\ell}.
\]
\end{cor}

\begin{proof}
Let $\ell \in \{1,2,\ldots, q-1\}$. Then using Lemma \ref{lem_t-partite}, we obtain
\begin{eqnarray*}
e(G[A]) & \ge &  \sum_{1 \le i < j \le q-1} n_in_j + \left(a - \sum_{i=1}^{q-1} n_i\right)  \sum_{i=1}^{q-1} n_i\\
               & \ge & n_{\ell} \sum_{1 \le i \le q-1, i \neq \ell} n_i + \left(a - \sum_{i=1}^{q-1} n_i\right) n_{\ell} = (a - n_{\ell}) n_{\ell}.
\end{eqnarray*}
\end{proof}

\begin{thm}
Let $n_1 \ge n_2 \ge \ldots \ge n_t$, where $t \ge 3$, $n_1, n_2 \ge k^2 - k + 1$ and either $n_3 \ge 2k-2$ or $n_3 \ge k-1$ and $n_4 \ge k-1$. Then $f(K_{n_1,n_2,\ldots, n_t}, k) = \sum_{i=3}^t n_i - 2k + 2$.
\end{thm}

\begin{proof}
Let $V = \cup_{i=1}^t V_i$ be the partition of the vertex set of $K_{n_1,n_2,\ldots, n_t}$ such that $|V_i| = n_i$. Let $A$ and $B$ be two vertex disjoint subsets of $V$ such that $\mad(A) \le 2k-2$ and $\mad(B) \le 2k-2$ and let $a_i = |A \cap V_i|$, $b_i = |B \cap V_i|$ for $1 \le i \le t$ and $|A| \ge |B|$. We will show that $|A \cup B| \le n_1 + n_2 + 2k -2$.

Suppose first that $|A| = n_1 + r$ for an integer $r > 0$. Since $|A| > n_1$, we can set $q \ge 2$ in Corollary \ref{cor_t-partite}  and so $e(G[A]) \ge (n_1 + r - n_1) n_1 = r n_1$. Further, as $\frac{2 e(G[A])}\le \mad(A) \le {n_1+r}\le 2k-2$, we have that
\[
r n_1 \le e(G[A]) \le (k-1)(n_1 + r),
\]
which, together with $n_2 \ge k^2 - k + 1$, implies that 
\[
r \le \frac{(k-1) n_1}{n_1-k+1} = k-1 + \frac{k^2 - 2k+1}{n_2-k+1} \le k-1 + \frac{k^2 - 2k+1}{k^2-2k+2}.
\]
Hence, $|A| \le n_1+ k-1$.

If $|B| \le n_1 + n_2 + 2k-2 - |A|$, it follows that $|A \cup B| \le n_1+n_2 + 2k-2$ and we are done. Hence assume that $|B| = n_2 + s$ for an integer $s \ge n_1 - |A| + 2k - 1 \ge k$. If $b_1 \le n_2$, it follows by Corollary \ref{cor_t-partite}, setting $n'_1 = n_2$, $n'_i = n_i$ for $i \ge 2$ and reordering the indices if necessary, that $e(G[B]) \ge s n_2$ and, analogously as above, we obtain that $s \le k-1$ and thus $|A \cup B| \le n_1 + n_2 + 2k-2$ and we are done. 

Assume now for contradiction that  $b_1 = n_2 + p$ for an integer $p$ with $s \ge p >0$. By  Corollary \ref{cor_t-partite}, setting $n'_1 = n_2 + p$, $n'_i = n_i$ for $i \ge 2$ and reordering the indices if necessary, it follows that $e(G[B]) \ge (n_2 + p) (s-p)$. Since $\frac{2 e(G[B])}{n_2+s} \le \mad(B) \le 2k-2$, we have that
\[
(n_2 + p) (s-p) \le e(G[B]) \le (k-1)(n_2 + s),
\]
which, together with $n_2 \ge k^2 - k + 1$, implies that 
\begin{eqnarray*}
s &\le& \frac{n_2(k-1)+ p (n_2 + p)}{n_2+p-k+1} = \frac{(n_2+p-k+1)(k-1) + p(n_2+p-k+1)+(k-1)^2}{n_2+p-k+1}\\
   &= & p+k-1 + \frac{(k-1)^2}{n_2+p-k+1} \\
   &\le& p+k-1 + \frac{(k-1)^2}{k^2-2k+2+p} =  p+k-1 + \frac{(k-1)^2}{(k-1)^2+p+1}.
\end{eqnarray*}
Hence, $s \le p+k-1$ and thus $p \ge s-k+1 \ge n_1 - |A| + k$. Now consider $e(G[A])$. Since $a_1 \le n_1 - b_1 \le n_1 - n_2 - p$, we have that $\sum_{i=2}^t a_i = |A| - a_1 \ge n_1 +k - p  -a_1 \ge n_2 + k$. Thus by Corollary \ref{cor_t-partite}, setting $n''_1 = n_1 -n_2 -p$ , $n''_i = n_i$ for $i \ge 2$ and reordering the indices if necessary, it follows that $e(G[A]) \ge (|A|-n_2) n_2$. Note that $k^2-k+1 \le n_2 \le n_1-p \le  |A| -k$ and thus $|A| \ge k^2+1$. Define now the function $f(x) = (|A|-x) x = |A| x - x^2$, where $k^2-k+1 \le x \le |A| -k$.  Evidently, $f(x)$ takes its minimum either when $x = k^2-k+1$ or when $x =  |A| -k$. Since $f(k^2-k+1) = |A|(k^2-k+1) - (k^2-k+1)^2$ and $f(|A|-k) = |A|k - k^2$, to prove that $f(k^2-k+1) \ge  f(|A| -k)$ is equivalent to
\[
|A| \ge \frac{(k^2-k+1)^2-k^2}{(k-1)^2} = \frac{((k-1)^2 +k)^2 - k^2}{(k-1)^2} = \frac{(k-1)^4 + 2k(k-1)^2}{(k-1)^2} = k^2+1,
\]
which is certainly true. Hence
 \[
 e(G[A]) \ge f(|A|-k) = (|A| - k) k \ge |A| (k-1) + |A| - k^2 \ge |A| (k-1) + 1,
 \]
 which contradicts the assumption that  $\frac{2 e(G[A])}{|A|} \le \mad(A) \le 2k-2$. Thus $b_1 > n_2$ is not possible.
 
Altogether it follows that $M(G,k-1) \le n_1 + n_2 + 2k -2$. To see that this bound is sharp, take $A = V_1 \cup A'$ and $B = V_2 \cup B'$, where $A'$ and $B'$ are disjoint sets such that $|A'|=|B'| = k-1$ and either $A' \cup B' \subseteq V_3$ in case that $n_3 \ge 2k-2$ or $A' \subseteq V_3$ and $B' \subseteq V_4$ in case that $2k-2 > n_3 \ge k-1$ and $n_4 \ge k-1$. Observe that $e(G[A]) = n_1 (k-1)$ and $e(G[B]) = n_2 (k-1)$ and thus $\mad(A), \mad(B) \le k-1$ and, since $|A \cup B| = n_1 + n_2 + 2k -2$, it realizes the bound.

Hence we have proven that $M(G,k-1) = n_1 + n_2 + 2k -2$ and thus, with Theorem \ref{f in M}, we obtain $f(G,k) = n - n_1 - n_2 - 2k +2 =  \sum_{i=3}^t n_i - 2k + 2$.
\end{proof}

A \emph{perfect} graph $G$ is defined as a graph such that any of its induced subgraphs $G^*$, including $G$ itself, has $\chi(G^*) = \omega(G^*)$, where $\omega(G^*)$ stands for the clique number of $G^*$. Let us denote by $h(G,k_3)$ the cardinality of a minimum set of vertices $T$ which hits all triangles of $G$, i.e. such that every triangle contained in $G$ has at least one vertex in $T$. More information about these ``triangle-hitting sets'' can be found under the concepts of ``clique covering'' or ``transversals in hypergraphs'', see \cite{Al, An, EGT, Tu}. Now we can state the following theorem.

\begin{thm}\label{perfect}
If $G$ is a graph, then $f(G,1) \ge  h(G,K_3)$. Moreover, if $G$ is perfect, then $f(G,1) =  h(G,K_3)$.
\end{thm}

\begin{proof}
Let $D$ be an orientation of $G$ realizing $f(G,1)$ and let $S$ be the set of vertices $v$ having $\deg^+(v) \ge 1$ and $\deg^-(v) \ge 1$ under orientation $D$. As evidently every triangle of $G$ has to have at least one vertex in $S$, the set $S$ hits all triangles of $G$ and thus $h(G,K_3) \le |S| = f(G,1)$.

Let now $G$ be perfect and let $T$ be a minimum set hitting all triangles of $G$. Then the graph $G^*= G - T$ is triangle-free and, as $G$ is perfect, $G^*$ is perfect as well and thus $\chi(G^*) = \omega(G^*) \le 2$. Hence $G^*$ is bipartite and we obtain $M(G,0) \ge |V(G^*)| = |V \setminus T| = n - h(G,K_3)$. With the inequality proven above for general graphs, it follows $f(G,1) =  h(G,K_3)$.
\end{proof}

A planar graph is called {\it maximal planar}, for short \emph{MP}, if the addition of any edge would destroy that property. A {\it maximal outerplanar graph}, abbreviated \emph{MOP}, is a triangulation of the polygon. By the Four Color Theorem, every MP graph is $4$-colorable. Also, it is well-known that every MOP graph is $3$-colorable. Moreover, every MP graph on $n \ge 3$ vertices has exactly $3n-6$ edges, while every triangle-free planar graph on $n \ge 3$ vertices has at most $2n - 4$ edges.

\begin{thm} Let $G$ be a graph.\\[-4ex]
\begin{enumerate}
\item[\rm{(1)}] If $G$ is a MP on $n \ge 4$ vertices, then $2 \le f(G,1) \le \frac{n}{2}$ and both bounds are sharp.
\item[\rm{(2)}] If $G$ is a MOP on $n \ge 3$, then $1 \le f(G,1) \le \frac{n}{3}$ and both bounds are sharp.
\end{enumerate}
\end{thm}

\begin{proof}
(1) Since there are MP graphs $G$ with $\chi(G) = 3$, using the chromatic lower bound, we get only $f(G,1) \ge 1$, which is exact for $n = 3$. For $n \ge 4$, assume there is an MP graph $G$ with $f(G,1) = 1$. Then by Theorem \ref{perfect} there is a vertex $v$ hitting al triangles of $G$ and thus $G^* = G-v$ is a triangle-free planar graph on $n-1 \ge 3$ vertices. Therefore, $e(G^*) \le 2(n-1) - 4 = 2n -6$. But now it follows that $n-1 \ge \deg(v) = e(G) - e(G^*) \ge 3n - 6 - (2n - 6) = n$, which is a contradiction. Therefore, $f(G,1) \ge 2$. To see the sharpness, consider the graph consisting of two adjacent vertices $u$ and $v$ and $n-2$ vertices forming a path $x_1x_2\ldots x_{n-2}$ such that both $u$ and $v$ are adjacent to every vertex on the path. Now orient the edges from $x_i$ out, when $i = 0 \,(\rm{mod} 2)$, and in, when $i = 1 \, (\rm{mod} 2)$, and lastly orient $u$ to $v$. Then $u$ and $v$ are the only vertices with indegree and outdegree at least $1$. 

As every planar graph is $4$-colorable, by Corollary \ref{coro_graphs}(3) we have $f(G,1) \le \lfloor \frac{(t-2)n}{t} \rfloor \le \frac{n}{2}$, where $\chi(G) = t \le 4$. To see the sharpness, take $n/4$ vertex disjoint $K_4$'s and join them by edges to a MP graph $G$. For this graph, every $K_4$ contributes with two vertices of indegree and outdegree at least $1$ and hence $f(G,1) \ge 2 n/4 = n/2$.

(2) Let $G$ be a MOP graph. Since $\chi(G) =3$, Theorem \ref{low_chi} implies $f(G,1) \ge \chi(G) - 2 = 1$ and hence the lower bound follows. This can be realized by a graph $G$ consisting of a cycle $C_n$ and such that one of its vertices, say $z$, is adjacent to all other vertices. Orient the edges of the cycle in such a way that every vertex but possibly $z$ (when $n$ is odd) has either indegree or outdegree $0$. The remaining edges are oriented either to or from $z$ in such a way that the indegree or the outdegree of the other vertices of the cycle remains being $0$.  Thus $z$ is the only vertex having $\deg^+(z) > 0$ and $\deg^-(z) >0$ and hence $f(G,1) =1$.

For the upper bound, since $\chi(G) = 3$, we have $f(G,1) \le \frac{n}{3}$ by Corollary \ref{coro_graphs}(3). To see the sharpness, take $n/3$ vertex disjoint $K_3$ and complete it by adding edges to a maximal outerplanar graph $G$. As $G$ contains $n/3$ disjoint $K_3$'s, we must have $f(G,1) \ge n/3$ and thus $f(G,1) = n/3$ must hold. Thus the upper bound is also sharp.
\end{proof}

\begin{thm} The following assertions hold:\\[-4ex]
\begin{enumerate}
\item[\rm{(1)}] There exist positive constants $c_1(k)$ and $c_2(k)$ such that 
\[2n - c_1(k) \log(n) \le \max\{f(G,k) + f(\overline{G},k) : |V(G)| = n\} \le 2n - c_2(k) \log n.\]
\item[\rm{(2)}] $f(G,1) + f(\overline{G},1) \ge n-4$ and this bound is sharp.
\item[\rm{(3)}] $n- 16 k + 12 \le \min\{f(G,k) + f(\overline{G},k) : |V(G)| = n\} \le n - 8k+4$.
\end{enumerate}
\end{thm}

\begin{proof}
(1) By Corollary \ref{coro_graphs} (1), for any graph $G$ on $n$ vertices, $n - \alpha(G) \ge f(G,1) \ge n - 2\alpha(G)$ for any $n$-vetex graph $G$. Since by the Ramsey-Theorem $\max\{\alpha(G), \alpha(\overline{G}) \} \ge c \log n$ and thus $\alpha(G) + \alpha(\overline{G}) \ge c \log n$ for a constant $c > 0$, we obtain $f(G,1) + f(\overline{G},1) \le n - \alpha(G) + n - \alpha(\overline{G}) \le 2n - c \log n$. Further, the Ramsey-Theorem guarantees also the existence of  a graph $G$ and a constant $c$ for which $\alpha(G) \sim \alpha(\overline{G}) \sim c \log n$. Hence for such graphs, $f(G,1) + f(\overline{G},1) \ge n - 2\alpha(G) + n - 2\alpha(\overline{G}) \ge 2n - 4 c \log n$. Hence we have that $2n - 4 c \log n \le \max\{f(G,1) + f(\overline{G},1) : |V(G)| = n\} \le 2n - c \log n$, which proves the theorem for $k=1$.

Since $f(G,k) \le f(G, k-1)$, we also obtain for the Ramsey-graphs given above that $f(G,k) + f(\overline{G},1) \le f(G,1) + f(\overline{G},1) \le 2n - c \log n$ and we can take $c_2(k) = c$. By Theorem \ref{Thm3}, we have that $f(G,k) \ge n- 2 \beta_{2(k-1)}(G)$, where $\beta_{2(k-1)}(G)$ is the cardinality of a maximum set of vertices of $G$ whose induced graph is $2(k-1)$-degenerate. Observe that for the above Ramsey-graphs, if $F$ is the largest induced $2(k-1)$-degenerate graph of $G$, then $c \log n \ge \alpha(G) \ge \alpha(F) \ge \frac{|F|}{2k-1} =  \frac{\beta_{2(k-1)}(G)}{2k-1}$ and hence $ \beta_{2(k-1)}(G) \le (2k-1) c \log n$. The same holds for $\overline{G}$, namely  $\beta_{2(k-1)}(\overline{G}) \le (2k-1) c \log n$. Hence the above Ramsey-graphs have $f(G,k) + f(\overline{G},k) \ge 2n - 4 (2k-1) c \log n$ and we may take $c_1(k) = 4c(2k-1)$.

(2) Recall that $f(G,1) = M(G,0)$. Let $A$ and $B$ be two disjoint independent sets realizing $M(G,0)$. Then in $\overline{G}$, $A$ and $B$ are two vertex disjoint cliques of cardinality, say, $a$ and $b$. Hence, $f(G,1) + f(\overline{G},1) \ge n - (a + b) + (a-2) + (b-2) = n-4$ (the inequality sign is because of the possible case that $a=1$ or $b=1$). To see that the bound can be attained, take $G = K_a \cup K_b$ and $\overline{G} = K_{a,b}$. Then $f(G,1) = f(K_a,1) + f(K_b) = a + b - 4 = n-4$, while $f(\overline{G},1) = f(K_{a,b},1) = 0$. Hence $f(G,1) + f(\overline{G},1) = n-4$ and the bound above is attained.

(3) Let $G$ be a graph and let $A$ and $B$ be two vertex disjoint sets realizing $M(G,0)$. Then $\mad_G(A) \le 2(k-1)$ and $\mad_G(B) \le 2 (k-1)$. Let also $X$ and $Y$ be two vertex disjoint sets realizing $M(\overline{G},0)$. Then $\mad_{\overline{G}}(X) \le 2(k-1)$ and $\mad_{\overline{G}}(Y) \le 2 (k-1)$. Now consider any set $T \subseteq V$ such that $\mad_{\overline{G}}(T) \le 2(k-1)$. Then $e(\overline{G}[T]) \le (k-1) |T|$ and hence in $G$ we have $e(G[T]) \ge {|T| \choose 2} - (k-1) |T|$. This implies that $\mad_G(T) \ge \frac{2 e(G[T])}{|T|} = (|T| - 1)-2(k-1)$ which is greater than $2(k-1)$ when $|T| > 4(k-1) + 1$. Hence, as $\mad_{\overline{G}}(A \cap X) \le \mad_{\overline{G}}(X)\le 2(k-1)$, we obtain that $|A \cap X| \le 4(k-1) +1$. The same happens with $B \cap X$, $A \cap Y$ and $B \cap Y$, and thus also $|B \cap X| \le 4(k-1) +1$, $|A \cap Y| \le 4(k-1) +1$ and $|B \cap Y| \le 4(k-1) +1$. Now we obtain, as $A$ and $B$ are vertex disjoint and $X$ and $Y$ as well,
\begin{eqnarray*}
f(G,k) + f(\overline{G},k) &=& n - |A \cup B| + n - |X \cup Y|\\
                                    &=&  = 2n - (|A \cup B \cup X \cup Y| + |A\cap X| + |A \cap Y| + |B \cap X| + |B \cap Y|)\\
                                    &\ge& 2n - (n + 4(4(k-1)+1))  = n - 16(k-1) - 4 = n - 16 k + 12.
\end{eqnarray*}
Let now $G = K_a \cup K_b$, where $a, b \ge 2$ and $\overline{G} = K_{a,b}$. Then by Theorem \ref{f(n,r,1,k)} $f(G,k) = f(K_a,k)+f(K_b,k) = a-4k+2 + b-4k+2 = a+b - 8k + 4 = n - 8k + 4$, while $0 \le f(\overline{G},k) \le f(\overline{G}, 1) = f(K_{a,b}, 1) = 0$. Hence for this graph we obtain $f(G,k) + f(\overline{G},k) = n-8k +4$. Altogether we obtain that $n-16k+12 \le \min \{f(G,k) + f(\overline{G},k)\} \le n - 8k+4$.
\end{proof}

\section{Ramsey type theorem for $f(H,p,k)$}

\subsection{Bounds using the notion of $\bp(H,p)$}

\begin{thm}\label{b(H,p)}
Let $H$ be an $r$-uniform hypergraph and $p$ an integer with $1 \le p \le r-1$. Then $f(H,p,1) \ge {n \choose p} - \bp(H,p)$ and equality holds for $p =1$ and $p = r-1$.
\end{thm}

\begin{proof}
Let $D$ be an orientation of the edges of $H$ that realizes $f(H,p,1)$. Let $A_i$ be the set of $p$-sets $A\subseteq V$ in which the $i$-th coordinate of its degree vector is the first coordinate equal to $0$, $1 \le i \le {r \choose p}$. Clearly, some or even all of the $A_i$ can be empty and they are pairwise disjoint as no $p$-set can be in both $A_i$ and $A_j$ by definition. As for every edge $E$, there is a $p$-subset of $E$ in each of the possible positions $i$, $1 \le i \le {r \choose p}$, not all $p$-sets contained in an edge can be contained in some $A_i$ and thus $\bigcup A_i$, $1 \le i \le {r \choose p}$, is a family of $p$-sets colored by $r \choose p$ colors such that no edge is $p$-monochromatic. Hence $b(H,p) \ge {n \choose p} - f(H,p,1)$.

We will prove now that equality holds when $p=1$ or $p=r-1$. Let $B(H, p) = \bigcup A_i$, $1\le i \le {r \choose p}$, be a largest family of $p$-sets using at most ${r \choose p}$ colors such that no edge $E$ of $H$ with all its $p$-sets in $B(H,p)$ is $p$-monochromatic, where $A_i$ are the color classes. Now we have to show that every edge $E \in E(H)$ can be oriented in such a way that if $E$ contains a $p$-set $A \in A_i$, then the vertices of $A$ will not be placed on the set of positions corresponding to the $i$-th coordinate of the degree vector. Thus we will obtain an orientation of the edges of $H$ with at least $\bp(H,p)$ $p$-sets with at least one zero-coordinate in its degree vector, showing that $f(H,p,1) \le {n \choose p} - b(H,p)$. Such an order is possible in case that $p=1$ or $p=r-1$.

(a) Let $p = 1$. Suppose that $E = \{v_1,v_2, \ldots, v_r\}$ and consider all the permutations of the vertices of $E$. The probability that a vertex $v$ is placed on a forbidden position is $1/r$ if $\{v\} \in \bigcup A_i$ and $0$ otherwise. Hence the expected number of vertices of $E$ placed in a forbidden position is $$\sum_{v \in E \cap (\bigcup A_i)} \frac{1}{r} \;\le \;1,$$
and equality holds when $E \subseteq \bigcup A_i$. So if $E$ is not contained in $\bigcup A_i$, the expected number of vertices placed in a forbidden position is less than $1$ and thus the required order exists, namely there is an order of $E$ such that no vertex ($p=1$) is placed in a forbidden position. Otherwise, if $E \subseteq \bigcup A_i$, as the $A_i$'s are independent (and no edge contained in $B(H,p)$ is $p$-monochromatic), there has to be a placement of $u \in A_i$ and $v \in A_j$, $i \neq j$ with $u , v \in E$ such that they are placed in the $i$-th and $j$-th positions and the rest of $E$ is placed arbitrarily among the color classes. This is an ordering of $E$ in which at least two vertices  are placed in forbidden positions. But, as the expected number of vertices in forbidden positions is equal to $1$, this implies that there has to be also an ordering of $E$ with no vertex placed into a forbidden position. Hence we have proved that there is an orientation of the edges such that all vertices in $A_i$ have a zero on its $i$-th coordinate of the degree vector. Thus there are at least $\bp(H,1)$ vertices whose degree vector contains a zero-coordinate, implying that $f(H,1,1) \le n - \bp(H,1)$, and with the inequality proved above, we obtain $f(H,1,1) = n - \bp(H,1)$.

(b) Let $p=r-1$. Suppose that $E = \{v_1,v_2, \ldots, v_r\}$ and consider all the permutations of the vertices of $E$. The probability that an $(r-1)$-set $A \subseteq E$ is placed on a forbidden position is ${r \choose r-1}^{-1} = 1/r$ if $A \in \bigcup A_i$ and $0$ otherwise. Hence the expected number of $(r-1)$-sets of $E$ placed in a forbidden position is 
$$\sum_{A \subseteq E \cap (\bigcup A_i), |A|=r-1} \frac{1}{r} \;\le \;1,$$
and equality holds precisely when all $(r-1)$-subsets of $E$ are contained in $\bigcup A_i$. So if not all $(r-1)$-subsets of $E$ are contained in $\bigcup A_i$, then the expected number of $(r-1)$-subsets of $E$ placed in a forbidden position is less than $1$ and hence the required order exists. Otherwise, i.e. if all $(r-1)$-subsets of $E$ are contained in $B(H,r-1) = \bigcup A_i$, as $E$ is not $(r-1)$-monochromatic, there are at least two $(r-1)$-subsets $X, Y \subseteq E$ in distinct color classes. Say that $X \in A_i$ and $Y \in A_j$. We want to show the existence of an orientation of $E$ such that $X$ and $Y$ are placed into the sets of positions corresponding, respectively, to the $i$-th and $j$-th coordinate of the degree vector. To do so for general $p$-sets, we need that the $p$-sets of positions $P_i$ and $P_j$  corresponding, respectively, to the $i$-th and $j$-th coordinates of the degree vector, are such that $|P_i \cap P_j| = |X \cap Y|$. This will allow us to orient $E$ such that $X$ has position set $P_i$ and $Y$ has position set $P_j$. For $p=r-1$, we have $|P_i \cap P_j| = r-2$ and $|X \cap Y| = r-2$ and such an orientation exists just by putting the vertices of $X \cap Y$ into positions $P_i \cap P_j$ and the vertex $v \in X \setminus Y$ into the position given by $P_i \setminus P_j$, while the vertex $u \in Y \setminus X$ into position $P_j \setminus P_i$. In this manner, we have produced an ordering of $E$ in which two $(r-1)$-sets are placed on forbidden positions. But, as the expected number of  $(r-1)$-sets of $E$ in forbidden positions is equal to $1$, this implies that there has to be also an ordering of $E$ with no $(r-1)$-set placed in a forbidden position. Hence we have proved that there is an orientation of the edges such that all  $(r-1)$-sets in $A_i$ have a zero on its $i$-th coordinate of the degree vector. Thus $f(H,r-1,1) \le n - \bp(H,r-1)$, and with the inequality proved above, we obtain $f(H,r-1,1) = n - \bp(H,r-1)$.
\end{proof}

\begin{cor} The following assertions hold:\\[-4ex]
\begin{enumerate}
\item[\rm{(1)}] $f(r,r-1,1) = R(H(r,r-1),r)$, the Ramsey number of $H(r,r-1)$ using $r$ colors.
\item[\rm{(2)}] $f(3,2,1) = 17$, $f(4,3,1) \le 15202$. 
\end{enumerate}
\end{cor}

\begin{proof}
(1) By Theorem \ref{b(H,p)}, $f(n,r,r-1,1) = {n \choose r-1} - \bp(H(n,r),r-1)$. Since $\bp(H(n,r))$ is the cardinality of a largest family of $(r-1)$-sets that can be colored with at most $r$ colors such that no edge of $H(n,r)$ is $(r-1)$-monochromatic, it follows that $f(r,r-1,1)$ is the minimum $n$ such that any coloring of the $(r-1$-sets of $V(H(n,r))$ has a $(r-1)$-monochromatic edge. Hence $f(r,r-1,1) = R(H(r,r-1),r)$.\\
(2) By (1), we obtain $f(3,2,1) = R(H(3,2),3) = R(K_3,K_3,K_3) = 17$ and $f(4,3,1) = R(H(4,3),1) \le 15202$ (see \cite{Ra}).
\end{proof}

\subsection{Bounds using Ramsey numbers}

\begin{thm}\label{ramsey}
There exists a constant $0 < c(r,p,k) < 1$ and a positive integer $N(r,p,k)$ such that if $H$ is an $r$-uniform hypergraph on $n \ge N(r,p,k)$ vertices and with $e(H) \ge c(r,p,k) {n \choose r}$ edges, then $f(H,p,k) > 0$.
\end{thm}

We will prove the following theorem from which Theorem \ref{ramsey} is a corollary from Turan's theorem for hyergraphs.

\begin{thm}\label{t-set}
Let $H = H(n,r)$ be the complete $r$-uniform hypergraph on $n$ vertices and let $r > p \ge 1$, $t \ge p, k \ge 1$. Then there is an integer $N(r, p ,t, k) > 0$ such that if $n \ge N(r, p, t, k)$ then every orientation $D(H)$ has a $t$-set $B \subseteq V(H)$ all of its $p$-subsets $S \subseteq B$ having $\deg_i(S) \ge k$ for all $1 \le i \le {r \choose p}$. 
\end{thm}

\begin{proof}
Define a coloring of the $p$-sets of $\{1,2, \ldots,r\}$ with colors $1,2, \dots, {r \choose p}+1$ the following way. If $S \subseteq \{1,2, \ldots,r\}$ is a $p$-set such that the $i$-th coordinate of the degree vector is the first to be at most $k-1$, then we assign $S$ the color $i$. If all of its coordinates are at least $k$, then $S$ gets color ${r \choose p} + 1$ assigned. Let $q$ be the smallest integer such that $(k-1) {q \choose p} < {q \choose r}$. By the Ramsey Theorem for hypergraphs, if $n$ is sufficiently large, i.e. $n \ge N(r,p,q,t,k) \ge R(q,q, \ldots, q, t)$, $q$ appearing ${r \choose p}$ times, then there is a $q$-set $Q$ with all its $p$-sets of the same color, for some color from $\{1,2, \ldots, {r \choose p}\}$, or there is a $t$-set $B$ all whose $p$-sets are colored with color ${r \choose p} + 1$. If the latter case occurs, we are done. So suppose this is not the case and there is a $q$-set $Q$ with all its $p$-sets of the same color, say $\ell$, which means that, for every $p$-set $S$ of $Q$, $\deg_{\ell}(S)$ is the first coordinate of the degree vector being less than $k$. By double-counting the edges of the subhypergraph induced by $Q$, we obtain
\[
{q \choose r} = e(H[Q]) \le \sum_{S \subseteq Q, |S|=p} \deg_{\ell}(S) \le (k-1) {q \choose p},
\]
which is a contradiction to the choice of $q$. Hence there has to exist a $t$-set $B$ all whose $p$-sets are colored with color ${r \choose p} + 1$, i.e. such that all their coordinates on the degree vectors are at least $k$.
\end{proof}

Now we can give the proof of Theorem \ref{ramsey}.

\noindent
{\it Proof of Theorem \ref{ramsey}.} 
Let $q$ be the smallest integer such that $(k-1) {q \choose p} < {q \choose r}$ and let $H$ be an $r$-uniform hypergraph on $n \ge N(r,p,k)$ vertices, where $N(r,p,k) > R(q,q, \ldots, q)$, $q$ appearing ${n \choose p}$ times. Let $n' = R(q,q, \ldots, q)$, $q$ appearing ${n \choose p}$ times. By the Turan-Theorem for hypergraphs, if $e(H)$ is suficently large, say $e(H) > c(r,p,k) {n \choose r}$ for a constant $0 < c(r,p,k) < 1$, then it contains a $H(n',r)$ as a subhypergraph. Let $D$ be an orientation of $H$ realizing $f(H,p,k)$. Suppose that in the induced orientation on $H(n',r)$, all $p$-sets of $\{1,2, \ldots, r\}$ have at least one coordinate of the degree vector less than $k$. Then, as in the proof of Theorem \ref{t-set}, we can color the $p$-sets of $H(n',r)$ with ${n' \choose p}$ colors according to the first coordinate which is less than $k$. But now the Ramsey-Theorem for hypergraphs implies that there is a $q$-set $Q$ with all of its $p$-sets colored with the same color, which by the choice of $q$ is not possible. Hence there has to be a $p$-set in the subhypergraph $H(n',r)$ with all its coordinates at least $k$. This implies that $f(H,p,k) \ge f(H(n',r),p,k) > 0$. $\Box$\\

In the next theorem we will show that, in every orientation of $H(n,r)$, a positive fraction of all the $p$-sets must have all degree-vector coordinates at least $k$, as $n$ increases. For this purpose, we need to define the concept of packing. An \emph{$(n,m,p)$-packing} $\mathcal{F}$ is a family of $m$-sets of $\{1,2, \ldots, n\}$ having the property that every of its $p$-sets is contained in at most one $F \in \mathcal{F}$. In \cite{Ro}, R\"odl proved that there exists a family $\mathcal{F}$ with
\[|\mathcal{F}| \ge {n \choose p}{m \choose p}^{-1} (1 - o(1)),\]
where $p$ and $m$ are fixed and $n$ tends to infinity.

\begin{thm}
For positive integers $n$, $r$, $p \le r$ and $k$,
\[f(n,r,p,k) \ge {n \choose p}{f(r,p,k) \choose p}^{-1} (1 - o(1)),\]
where $r$, $p$ and $k$ are fixed and the	$o(1)$-term tends to zero as $n$ tends to infinity.
\end{thm}

\begin{proof}
Setting $m = f(r,p,k)$, in every orientation of $H(m,r)$ there is a $p$-set with all its degree-vector coordinates at least $k$. Let $\mathcal{F} = \{F_1, F_2, \ldots, F_t\}$ be a maximum $(n,m,p)$-packing. Consider every set $F_i$ with all its $p$-sets as a copy of $H(m,p)$ into $H(n,p)$, $1 \le i \le t$. Thus, any two copies of $H(m,p)$ have at most $p-1$ vertices in common. Now consider an orientation of $H(n,r)$ and the induced orientation on the copies of $H(m,r)$ defined on the vertex sets $F_1, F_2, \ldots, F_t$. Clearly, no $p$-set belongs to two copies of $H(m,p)$ defined on the same ground set of vertices. Also, as $p \le r$, no $r$-edge belongs to two copies of $H(m,r)$. Hence, in this orientation, each copy of $H(m,p)$ contributes with at least one $p$-set with all its degree-vector coordinates at least $k$. Since all these sets are distinct, it follows that $f(n,r,p,k) \ge t$.
Then, by R\"odl's Theorem, we obtain
\[f(n,r,p,k) \ge t \ge {n \choose p}{m \choose p}^{-1} (1 - o(1)),\]
where $m$ and $p$ are fixed and $n$ grows.
\end{proof}
 
 \section{Complexity of f(G,k)}
 
 Given a graph $G$ on $n$ vertices and a positive integer $k$, let us define the following problems.
 
 \noindent
 {\sc Independent Set}\\
 Does $G$ have an independent set $I$ with $|I| \ge k$?\\[2ex]
 \noindent
 {\sc Two Disjoint Independent Sets}\\
  Does $G$ have two disjoint independent sets $A$ and $B$ with $|A \cup B| \ge 2k$?\\[2ex]
 \noindent
 {\sc $f(G,1)$-orientation} \\
 Does $G$ have an orientation with at most $n-k$ vertices with indegree and outdegree at least $1$?
 
 It is well-known that the problem  {\sc Independent Set} is NP-complete.
 
 \begin{thm}
The problem {\sc $f(G,1)$-orientation} is NP-complete. 
 \end{thm}
 
 \begin{proof}
Since $f(G,1) = n - M(G,0)$, to determine if $f(G,1) \le n-k$ is equivalent to prove that $M(G,0) \ge k$. Hence, we will show that the problem {\sc $f(G,1)$-orientation} is NP-complete by reducing the problem {\sc Independent Set} into the problem {\sc Two Disjoint Independent Sets}. Indeed given a graph $G$ on $n$ vertices and two vertex disjoint independent sets $A$ and $B$ claimed to have $|A \cup B| \ge 2k$, verification of independence, vertex-disjointness and cardinality of the union are easily checked in $O(n^2)$. Let $G$ and $k$ be an instance of the problem {\sc Independent Set}. Consider the graph $H = G \times K_2$ consisting of two copies of $G$ and all the edges between them. If $G$ has an independent  set $I$ of cardinality $|I| \ge k$, then $H$ has two disjoint independent sets whose union has cardinality  $2|I| \ge 2k$. Conversely, let $A$ and $B$ be two vertex disjoint independent sets of $H$ with $|A \cup B| \ge 2k$. Without loss of generality, say that $|A| \ge |B|$ and thus $|A| \ge k$. As, in $H$, every vertex of one copy of $G$ is adjacent to every vertex of the other copy of $G$, $A$ has to be fully contained in one of the two copies. Hence, $A$ is an independent set of $G$ with $|A| \ge k$. Thus, we have the desired reduction and, as the problem {\sc Independent Set} is NP-complete, it follows that the problems {\sc Two Disjoint Independent Sets} as well as {\sc $f(G,1)$-orientation} are NP-complete.
 \end{proof}
 
\section{Open problems}

When $G$ is a maximal planar graph on $n \ge 4$ vertices, we have shown that $2 \le f(G,1) \le n/2$, and both bounds are sharp. Since planar graphs are $5$-degenerate then, by Lemma \ref{lemma4} (2), $V(G)$ can be partitioned into two vertex disjoint sets $V_1$ and $V_2$ such that their induced subgraphs are $2$-degenerate. By Lemma \ref{lemma4} (1), $\mad(V_1) \le 4$ and $\mad(V_2) \le 4$ and hence $M(G,2) = n$ and $f(G,3) = n - M(G,2)=0$.  So every planar graph has an orientation for which every vertex $v$ has either $\deg^+(v) \ge 2$ or $\deg^-(v) \ge 2$. We are left with the problem of $f(G,2)$ for MP graphs, which is equivalent to the following problem: 

\begin{prob}
Does there exist an MP graph $G$ such that $V(G)$ can be partitioned into two vertex disjoint sets $A$ and $B$ such that $\mad(A) \le 2$ and $\mad(B) \le 2$? 
\end{prob}
 
There are examples of planar graphs which do not have such a partition with $\mad(A) < 2$ and $\mad(B) < 2$. For this, see the literature on the vertex arboricity of maximal planar graphs, for instance \cite{ChK, RW, Wo}. 
 
An easy upper bound for $f(G,2)$ comes from the fact that an MP graph $G$ can be vertex-partitioned into three subsets, each of one induces a forest (i.e. a $1$-degenerate subgraph). Let $V(G) = A  \cup B  \cup C$ be the such a partition, where $|A| \ge |B| \ge |C|$. Then $M(G,1) \ge 2n/3$ and hence $f(G,2) = n - M(G,1) \le n/3$. Another observation is that if there is a planar graph $G$ on $m$ vertices with $f(G,2) = q > 0$, then we can take $r$ copies of $G$ and complete them to a MP graph $H$ on $rm$ vertices. Then each copy of $G$ gives $q$ vertices with either $\deg^+(v) \ge 2$ or $\deg^-(v) \ge 2$. Hence $f(H,2) \ge rq = \frac{qn}{m}$. This gives that $\max \{f(G,2) \,:\, G \mbox{ n-vertex MP graph}\} \ge \frac{q}{m}n$. Hence either $\max \{f(G,2) \,:\, G \mbox{ n-vertex MP graph}\} = 0$ or there exists a constant $c > 0$ such that $\max \{f(G,2) \,:\, G \mbox{ n-vertex MP graph}\} \ge cn$. 
 
\begin{prob} 
Prove or disprove that $\max \{f(G,2) \,:\, G \mbox{ n-vertex MP graph}\} = 0$.
\end{prob}
 
\begin{prob} 
Find lower and upper bounds on $f(G,k)$ where $G$ belongs to the following families of graphs: claw-free graphs, $K_{1,r}$-free graphs, Line graphs, $K_3$-free graphs, series-parallel graphs, $k$-trees, etc.
 \end{prob}

\begin{prob}
Improve the bound $f(H,1,k) \ge  X(H) -  r(r(k-1) +1)$ or show sharpness  for $r> 2$.
\end{prob}
 
 \begin{prob}
Improve the bounds on $f(n,r,p,k)$ given in Section 4. 
\end{prob}

\begin{prob} 
Is it true that  $f(H,p,1)  = {n \choose p}  - \bp(H,p)$ for $2 \le  p \le r$? 
\end{prob}

\begin{prob}
Can we prove at least  $f(n,r,p,1)  = {n \choose p}  - \bp(H(n,r),p)$ for $2 \le  p \le r$?
\end{prob}

\begin{prob}
Is it possible to further generalize Hakimi's theorem from the case $r \ge 2$,  $p=1$ to  $r \ge p \ge 2$?
 \end{prob}

\end{document}